%% file: lcsys.tex
\documentclass[journal,twoside,web]{ieeecolor}
\usepackage{lcsys}
\usepackage{cite}
\usepackage{amsmath,amssymb,amsfonts}
\usepackage{algorithmic}
\usepackage{graphicx}
\usepackage{textcomp}

\input{header.tex}

\def\BibTeX{{\rm B\kern-.05em{\sc i\kern-.025em b}\kern-.08em
    T\kern-.1667em\lower.7ex\hbox{E}\kern-.125emX}}
\begin{document}
\title{Local Linear Convergence of Infeasible Optimization with Orthogonal Constraints}

\author{Youbang Sun$^{1}$, Shixiang Chen$^{2}$, Alfredo Garcia$^{3}$ and Shahin Shahrampour$^{1}$
\thanks{ This work is supported in part by NSF ECCS-2240788 and NSF ECCS-2240789 Awards. }
\thanks{$^{1}$ Y. Sun and S. Shahrampour are with the Department of Mechanical and Industrial Engineering at Northeastern University, Boston, MA 02115, USA. 
{\tt\small email:\{sun.youb,s.shahrampour\}@northeastern.edu}.}
\thanks{$^{2}$  S. Chen is with the School of Mathematical Sciences and Key Laboratory of the Ministry of Education for Mathematical Foundations and Applications of Digital Technology, University of Science and Technology of China, Hefei, 230027, P. R. China.
{\tt \small email: shxchen@ustc.edu.cn}.}
\thanks{$^{3}$  A. Garcia is with the Department of Industrial and Systems Engineering at Texas A\&M University, College Station, TX 77845, USA. 
{\tt \small email: alfredo.garcia@tamu.edu}.}
}

\maketitle
\thispagestyle{empty}

\begin{abstract}
Many classical and modern machine learning algorithms require solving optimization tasks under orthogonality constraints. Solving these tasks with feasible methods requires a gradient descent update followed by a retraction operation on the Stiefel manifold, which can be computationally expensive. Recently, an infeasible retraction-free approach, termed the landing algorithm, was proposed as an efficient alternative. Motivated by the common occurrence of orthogonality constraints in tasks such as principle component analysis and training of deep neural networks, this paper studies the landing algorithm and establishes a novel linear convergence rate for smooth non-convex functions using only a local Riemannian PŁ condition. Numerical experiments demonstrate that the landing algorithm performs on par with the state-of-the-art retraction-based methods with substantially reduced computational overhead.
\end{abstract}



\input{main_body}

\vspace{-2mm}

\bibliography{ref}
\bibliographystyle{plain}

\end{document}

%% file: header.tex
\usepackage{amsmath,amssymb,amsfonts,dsfont,bm, bbm}
\usepackage{color}
\usepackage{dblfloatfix}    

\usepackage{graphicx}
\usepackage{caption}
\usepackage{color}

\usepackage{subcaption}
\DeclareCaptionLabelFormat{withfigure}{(\arabic{figure}#2)}

\newtheorem{theorem}{Theorem}[section]
\newtheorem{proposition}[theorem]{Proposition}
\newtheorem{lemma}[theorem]{Lemma}

\newtheorem{definition}[theorem]{Definition}
\newtheorem{assumption}[theorem]{Assumption}

\usepackage{makecell} 
\usepackage{tablefootnote}
\usepackage{float}
\usepackage{graphicx}

\usepackage{algorithm, algorithmic}

\newcommand{\cL}{\mathcal{L}}

\newcommand{\R}{\mathbb{R}}

\newcommand{\be}{\begin{equation}}
\newcommand{\ee}{\end{equation}}
\newcommand{\ba}{\begin{array}}
\newcommand{\ea}{\end{array}}
\newcommand{\bad}{\begin{aligned}}
\newcommand{\ead}{\end{aligned}}

\newcommand{\normfro}[1]{\| #1 \|_{\text{F}}}

\newcommand{\sk}{\mathrm{skew}}

\newcommand{\sym}{\mathrm{sym}}

\newcommand{\St}{\mathrm{St}}

 \newcommand{\grad}{\mathrm{grad}}

\def\norm#1{\|{#1}\|}

%% file: main_body.tex
\vspace{-2mm}
\section{Introduction}
\label{sec:introduction}

Orthogonality constraints naturally appear in many machine learning problems, ranging from the classical principal component analysis (PCA) \cite{hotelling1933analysis} and canonical correlation analysis (CCA) \cite{hotelling1935canonical} to decentralized spectral analysis  \cite{kempe2008decentralized}, low-rank matrix approximation \cite{kishore2017literature} and dictionary learning \cite{raja2015cloud}.  More recently, due to the distinctive properties of orthogonal matrices,  
orthogonality constraints and regularization methods have been used for training deep neural networks \cite{arjovsky2016unitary,vorontsov2017orthogonality}, providing improvements in model robustness and stability \cite{trockman2021orthogonalizing}, as well as for adaptive fine-tuning in large language models \cite{zhang2023adaptive}. Optimization with orthogonality constraints is typically formulated as,
\vspace{-2mm}

\begin{equation}\label{eq:cen_opt}
\begin{aligned}
    \min_{x \in \mathbb{R}^{d\times r}} &f(x),\\
    \text{s.t.} \quad x\in \St(d, r) & \triangleq \{x \in \mathbb{R}^{d\times r}| x^\top x = I_r\},
\end{aligned}
\vspace{-1mm}
\end{equation}
where $\St(d,r)$ is referred to as the {\it Stiefel} manifold.

The Stiefel manifold can be seen as an equality constraint that is a non-convex set. As a result, traditional first-order optimization algorithms (e.g., gradient descent) fail to solve Problem \eqref{eq:cen_opt} directly. The topic of functional constrained optimization \cite{boob2023stochastic} has recently emerged; although these works address tasks with non-convex constraints, they mostly require hierarchical algorithms for convergence. Another common approach is to consider a Riemannian optimization algorithm, where in each iteration, the algorithm takes a Riemannian gradient step instead of a Euclidean one. 
Then, the algorithm performs a retraction operation to ensure feasibility and stay on the manifold \cite{absil2012projection}.
It has been shown by the previous literature that these retraction-based algorithms can achieve similar iteration complexity rates \cite{Boumal2016,zhang2016first} compared to their Euclidean counterparts.

While achieving favorable iteration complexity and satisfying the feasibility constraint in \eqref{eq:cen_opt}, these methods are computationally expensive due to the retraction operation in each iteration. A typical retraction operation on the Stiefel manifold usually requires calculating a matrix exponential, inversion, or square root. This means that a typical retraction requires $\mathcal{O}(dr^2)$ or more algebraic operations. When $r$ is large, the calculation time for retractions could dominate the runtime of Riemannian algorithms. In addition, matrix inversion or exponentiation have generally not been well-suited to execute on GPUs until recently \cite{adil2022first}, making the application of \eqref{eq:cen_opt} in modern machine learning computationally prohibitive.

As a mitigation, recent works have proposed \textit{infeasible} approaches to solve manifold optimization problems. These methods do not require explicit retraction operations (or other projections) and are often referred to as \textit{retraction-free} methods. While retraction-free optimization iterates are not strictly feasible, they gradually converge to the feasibility region.

In this work, we focus on a recently proposed retraction-free (or infeasible) method for solving Problem \ref{eq:cen_opt}, called the {\it landing algorithm} \cite{ablin2022fast, ablin2023infeasible}. Although \cite{ablin2023infeasible} provided a global finite-time convergence guarantee for this algorithm, the numerical results suggested that the actual convergence speed is much faster. Our work seeks to study this empirical observation and provide a theoretical justification for the faster convergence while exploring the possible assumptions required on the objective function in order to achieve this convergence guarantee. To this end, 
\begin{enumerate}
    \item In Section \ref{sec:formulation}, we introduce a background for the optimization on the Stiefel manifold and provide the retraction-free landing algorithm.
    \item In Section \ref{sec:analysis}, we flesh out the set of assumptions on the objective function, notably the local Riemannian Polyak-Łojasiewicz (PŁ) condition on the Stiefel manifold. 
    We then introduce the merit function for our analysis and provide a {\it local linear} convergence guarantee for the landing algorithm.
    \item In Section \ref{sec:experiments}, we present numerical experiments and demonstrate the advantage of the landing algorithm in terms of convergence speed and efficiency in a high-dimensional PCA task as well as training convolutional neural networks (CNNs).
\end{enumerate}

\subsection{Related Literature}
\subsubsection{Retraction-Based Riemannian Optimization}
The classical optimization algorithms under Riemannian geometry rely on the diffeomorphic nature between a Riemannian manifold and the Euclidean space. 
A classical approach is the line-search method; this includes Armijo line search and accelerated line search \cite{absil2008optimization}. 
However, these approaches are rarely used in modern machine learning with large datasets, since multiple evaluations along a geodesic are often undesirable. 
As a mitigation, many methods such as Riemannian gradient descent \cite{zhang2016first} use retractions to ensure the feasibility of the iterates. 
With both retractions and parallel transports, a wide range of optimization algorithms have been adapted to the manifold context. These methods include Riemannian accelerated gradient \cite{zhang2018towards}, Riemannian conjugate gradient \cite{sato2016dai}, and adaptive Riemannian gradient methods  \cite{becigneul2018riemannian}. 

\subsubsection{Retraction-Free Approaches}
In order to ensure the feasibility of the manifold constraint, the previously mentioned algorithms still require costly operations, such as retractions and parallel transports. Recently, several approaches have been proposed in an attempt to reformulate the constrained problem in \eqref{eq:cen_opt} into equivalent unconstrained problems, which can be solved efficiently. 
\cite{schechtman2023orthogonal} introduced ODCGM, using a non-smooth penalty to ensure that stationarity is strictly recovered in the alternative merit function. 
Another approach stems from Fletcher's penalty function \cite{fletcher2002nonlinear} and implements a smooth merit function.
\cite{gao2019parallelizable} proposed a proximal linearized augmented Lagrangian algorithm (PLAM) and solved the optimization problem with an augmented Lagrangian method. 
\cite{xiao2022class} proposed PenC, an exact penalty function model, and implemented a projected inexact gradient algorithm. PenC was later extended to a constraint dissolving function (CDF) approach in \cite{xiao2024dissolving}, with examples covering a wide range of manifolds. In this work, our focus is on proving fast local convergence for the landing algorithm \cite{ablin2022fast, ablin2023infeasible} as it provides a simple approach to retraction-free optimization on the Stiefel manifold, and the empirical evidence in \cite{ablin2023infeasible} suggested fast convergence speed around stationary points.

\section{Problem Formulation}\label{sec:formulation}

\subsection{Notations}\label{sec:notations}

We start this section with definitions and notations used in this manuscript. 
We denote the transpose of vector $x$ (and matrix $A$) as $x^\top$ (and $A^\top$), respectively.
The Frobenius norm of matrix $A$ is denoted as $\normfro{A}$. The skew and symmetric parts of a square matrix $A$ are denoted as $\sk(A) = \frac{1}{2}(A-A^\top)~\text{and}~\sym(A) = \frac{1}{2}(A+A^\top)$, respectively. The matrix inner product of matrices $A$ and $B$ is defined as $\langle A, B\rangle \triangleq Tr(AB^\top)$. 
We denote the identity matrix of rank $r$ as $I_r$. Given a set $\mathcal{S}$, we represent $\mathcal{D}(\mathcal{S}, \delta) \triangleq \{x \mid \text{dist}(\mathcal{S}, x) \leq \delta\}$, where $\text{dist}(\mathcal{S}, x)=\min_{y\in\mathcal{S}} \normfro{x-y}$. 

\subsection{Riemannian Optimization}
In the Euclidean space, the function gradient $\nabla f(x)$ is calculated as the differential of $f$ evaluated at $x$. The classical gradient descent (GD) algorithm seeks to minimize the objective function by taking a step towards the negative Euclidean gradient, and the iterate is defined as $x_{k+1} = x_k - \alpha \nabla f(x_k)$, with $\alpha$ denoting the step size.
On the other hand, the Riemannian gradient $\grad f(x)$ is defined on the tangent space of a point $x$ on the manifold with respect to a given metric. 
For the Stiefel manifold specifically, the Riemannian gradient with respect to the canonical metric \cite{edelman1998geometry} is given by the following closed-form expression:
\begin{equation}\label{eq:relative-gradient}
    \grad f(x) = \sk(\nabla f(x) x^\top)x.
\end{equation}
The Riemannian GD uses $\grad f(x)$ to perform a retraction operation on the iterate, which allows the iterate to move while staying on $\St(d, r)$; the Riemannian GD algorithm updates the iterates as
$$x_{k+1} = \mathcal{R}_{x_k}(-\alpha \grad f (x_k)),$$
where $\mathcal{R}_{x}: \mathcal{T}_{x} {\St(d,r)} \to \St(d,r)$ denotes the retraction operation, which maps vectors in the tangent space at point $x$ to points on the Stiefel manifold. Though there are multiple choices of retraction operations on the Stiefel manifold, including exponential retraction, Caylay retraction, and projection retraction, all of them require high computational costs. To address this issue, many recent application-focused works such as deep learning-related studies \cite{bansal2018can}  use a penalty function instead and aim to minimize $f(x) + \lambda \mathcal{N}(x)$, where $\mathcal{N}(x)$ penalizes the constraint violation. Although these methods are conceptually simple and have relatively low computational complexity, they only find an approximately accurate solution, and the end product is often infeasible.

\subsection{Retraction-Free Algorithm}
Contrary to the existing retraction-based and penalty-based algorithms, this paper focuses on the landing algorithm  proposed by \cite{ablin2022fast} as an infeasible, retraction-free numerical solver for  Problem \eqref{eq:cen_opt}. Unlike retraction-based algorithms, the iterates of the landing algorithm do not strictly stay on the Stiefel manifold, avoiding high computational costs. The landing algorithm also differs from conventional penalty methods as it converges to a critical point on the manifold.

The landing algorithm updates are based on the landing field, defined as
\begin{equation}\label{eq:landing_alg}
    \begin{aligned}
        \Lambda(x) &\triangleq \grad f(x) + \lambda \nabla p(x),
    \end{aligned}
\end{equation}
where the penalty function $p(x)\triangleq \frac{1}{4} \normfro{x^\top x - I_r}^2$ and $\lambda>0$ is a constant.
The exact update for the landing algorithm is provided in Algorithm \ref{alg:retraction_free}, where $\alpha$ denotes the step size. Though the iterate $x_k$ may not belong to $\St(d, r)$, $\grad f(x_k)$ is still well-defined as in \eqref{eq:relative-gradient} and is referred to as the {\it relative} gradient.

\begin{algorithm}[t]
	\caption{Retraction-free Landing Algorithm} 
 \label{alg:retraction_free}
	\begin{algorithmic}[1]
		\STATE{{\bf Input:} initial point $x_0 \in \St(d,r)^\epsilon$, $\alpha >0 , \lambda > 0$, $\epsilon\in(0,\frac{3}{4}).$ }  
		\FOR{$k=0,1,\ldots$} 
		\STATE{Calculate the landing field $\Lambda(x)$ and update $x$ as:
  \begin{equation*}
  \begin{aligned}
      \Lambda(x_k) &= \grad f(x_k) + \lambda x_k(x_k^\top x_k - I_r),\\
      x_{k+1} &= x_k -  \alpha \Lambda(x_k).
  \end{aligned}
  \end{equation*}
  }
\ENDFOR
\end{algorithmic}
\end{algorithm}

Observe that when $x$ is on the manifold, the penalty function $p(x)$ and its gradient $\nabla p(x)$ are equal to zero; hence, $\Lambda(x) = \grad f(x)$ for $x \in \St(d, r)$.
In addition, the two terms of the landing field are orthogonal to each other, since $\langle \grad f(x), \nabla p(x)\rangle = \langle \sk(\nabla f(x) x^\top), x(x^\top x -I_r)x^\top\rangle = 0$ \cite{ablin2023infeasible}. Therefore, $\Lambda(x) = 0$ if and only if $\nabla p(x)=0$ and $\grad f(x)=0$, which shows that $\Lambda(x) = 0$ when $x$ is a first-order stationary point of $f(x)$ on the manifold.

While the landing algorithm allows the iterates to move outside of the manifold, one must ensure that they stay close enough to the manifold within a {\it safety region}.
\begin{definition}[Safety Region \cite{ablin2023infeasible}]
With $\epsilon\in(0,3/4)$, we define the safety region as
\begin{equation}\label{eq:safe_region}
    \St(d,r)^\epsilon \triangleq \{x \in \R^{d\times r}| \normfro{x^\top x - I_r} \leq \epsilon\}.
\end{equation}
\end{definition}
It has been shown that for the landing field $\Lambda(x)$ defined in \eqref{eq:landing_alg}, a safety step size exists such that the algorithm never leaves the safety region.
\begin{proposition}[Safety Step Size \cite{ablin2023infeasible}]\label{prop:safe_cen}
    Let $x_k \in \St(d,r)^\epsilon$. Assuming that $\normfro{\grad f(x_k)} \leq G$, the next iterate satisfies $x_{k+1} \in \St(d,r)^\epsilon$ if the step size $\alpha$ satisfies
    \begin{equation*}
        \alpha \leq \alpha_{safe} \triangleq \min\bigg\{\frac{\lambda \epsilon(1-\epsilon)}{G^2 + \lambda^2(1+\epsilon)\epsilon^2}, \sqrt{\frac{\epsilon}{2G^2}}, \frac{1}{2\lambda}\bigg\}.
    \end{equation*}
\end{proposition}

\section{Theoretical Analysis}\label{sec:analysis}
In this section, we first state our theoretical assumptions and then introduce the merit function for our analysis. Next, we provide our main theoretical result, which is the local linear convergence of the landing algorithm.
\subsection{Technical Assumptions}
Let us start by introducing the following assumptions on the objective function in \eqref{eq:cen_opt}. These assumptions are considered to be standard in the literature \cite{zhang2016first, han2023riemannian} and are necessary for our improved convergence guarantee.
\begin{assumption} [Lipschitz Smoothness]
\label{assump:Lip}
We assume that the function $f: \mathbb{R}^{
d\times r} \rightarrow \mathbb{R}$ is twice continuously differentiable and $L$-smooth, i.e., for all $x, y \in \mathbb{R}^{d\times r}$, we have
\begin{equation}
    f(y) \leq f(x) + \langle\nabla f(x),y - x\rangle + \frac{L}{2}\normfro{y-x}^2.
\end{equation}
\end{assumption}
Different from the retraction-based algorithms, since the iterates can go outside the manifold in the infeasible algorithms, we assume the smoothness of the function in the entire Euclidean space instead of the Stiefel manifold. We also note that the smoothness is defined with respect to the Euclidean gradient instead of the Riemannian gradient.

Next, we discuss the local gradient growth of the non-convex function $f(x)$ in the sense of Polyak-Łojasiewicz, formally defined below.
\begin{assumption} [Local Riemannian PŁ Condition]\label{assump:PL}
The objective function $f(x) : \mathbb{R}^{
d\times r} \rightarrow \mathbb{R}$ satisfies the local Riemannian Polyak-Łojasiewicz (PŁ) condition on the Stiefel manifold with a factor $\mu>0$ if 
\begin{equation}
    |f(x) -f_{\mathcal{S}}^*| \leq \frac{1}{2\mu} \normfro{\grad f(x)}^2,
\end{equation}
for any point $x \in \St(d,r) \cap \mathcal{D}(\mathcal{S},  2\delta)$, where $\mathcal{S}$ denotes the set of all local minima with a given value $f_{\mathcal{S}}^*$.
\end{assumption}

Different from Assumption \ref{assump:Lip}, the assumption on the PŁ condition is strictly on the manifold, and the condition deals with the Riemannian gradient $\grad f(x)$ instead of the Euclidean gradient $\nabla f(x)$. 

The local Riemannian PŁ condition is less restrictive compared to other commonly studied scenarios. It is a relaxation of the global Riemannian PŁ condition, which could be restrictive in this case since the Stiefel manifold has a positive curvature. This assumption is also easier to satisfy compared to the geodesic strong convexity. Many traditional tasks involving the Stiefel manifold, such as the PCA problem and the generalized quadratic problem, satisfy the local PŁ condition but do not meet the other two conditions mentioned above.
These are sometimes referred to as quadratic problems and have been shown to satisfy Assumption \ref{assump:PL}, demonstrated in \cite{liu2019quadratic}.

\subsection{Merit Function}
To better understand the dynamics of the landing algorithm, we consider the following merit function \cite{ablin2023infeasible}, which is defined with respect to the objective in \eqref{eq:cen_opt} as
\begin{equation}\label{eq:merit_def}
\begin{aligned}
    \cL(x) &\triangleq f(x) + h(x) + \gamma p(x) ,\\
    \text{and} \quad h(x) &\triangleq -\frac{1}{2}\langle \sym(x^\top \nabla f(x)), x^\top x - I_r\rangle.
\end{aligned} 
\end{equation}
From the definition of $\cL$, one can prove that for any point $x\in\St(d,r)$, $\nabla \cL(x) = \nabla f(x)-x~\sym(x^\top \nabla f(x))$ (see Eq. (44) in \cite{ablin2023infeasible}). 
Similar to Proposition 6 of \cite{ablin2023infeasible}, for a given $\epsilon \in (0, 3/4)$, $\gamma$ is required to satisfy the following condition:
\begin{equation}\label{ineq:gamma}
\begin{aligned}
 \gamma \geq & \frac{2}{3-4\epsilon}\left(L(1-\epsilon) + 3s + \hat{L}^2 \frac{(1+\epsilon)^2}{\lambda(1-\epsilon)}\right),\\
    \text{where}~~~~s &= \sup_{x\in \St(d,r)^\epsilon} \normfro{\sym(x^\top \nabla f(x))},~~~\text{and}\\
    \hat{L} &= \max(L, \max_{x\in \St(d,r)^\epsilon} \normfro{\nabla f(x)}).
\end{aligned}
\end{equation}
The following proposition on the properties of $\cL(x)$ is also given in \cite{ablin2023infeasible} and is useful for our technical analysis.
\begin{proposition}\label{prop_L}
The merit function $\cL(x)$ satisfies the following properties.
\begin{enumerate}
    \item  $\cL(x)$ is $L_{\cL}$-smooth on $x \in \St(d, r)^\epsilon$, with $L_{\cL} \leq L_{f+h}+ (2+3\epsilon)\gamma$, where $L_{f+h}$ is the smoothness  of $f+h$.
    \item For $\rho = \min\{\frac{1}{2}, \frac{\gamma}{4\lambda(1+\epsilon)}\}$ and $x \in \St(d, r)^\epsilon$, we have 
    \begin{equation*}
        \langle \Lambda(x), \nabla \cL (x) \rangle \geq \rho \normfro{\Lambda(x)}^2.
    \end{equation*}
\end{enumerate}
\end{proposition}
The first property in Proposition \ref{prop_L} shows that $\cL$ is indeed a smooth function. The second property establishes a relationship between the landing field and the gradient of the merit function. Since $\langle \Lambda(x), \nabla \cL (x) \rangle \geq 0$, the landing field always points towards an ascent direction for the merit function. Therefore, a landing step can be seen as a descent step on the merit function. Moreover, since $\langle \grad f(x), \nabla p(x)\rangle=0$, the second property shows that within the neighborhood $\St(d,r)^\epsilon$, 
if $\nabla \cL(x) = 0$, we have $\nabla p(x)=0$ and $\grad f(x)=0$.

In addition, due to the Lipschitz smoothness of $f$, it is easy to show that $\Lambda(x)$ is first-order Lipschitz continuous; we denote its Lipschitz constant as $L_{\Lambda}$. $L_\Lambda$ and $L_\cL$ can be set to be on the same order as $L$ given ideal $\lambda$ and $\gamma$. For the ease of analysis below, we denote $L' \triangleq \max \{\hat{L}, L_\Lambda, L_\cL\}.$

\subsection{Local Linear Convergence of Landing Algorithm}
In this section, we study the convergence of the landing algorithm under the local Riemannian PŁ condition. Since this paper focuses on the local convergence guarantees of the landing algorithm, and that Assumption \ref{assump:PL} is only applicable to $x \in \St(d,r) \cap \mathcal{D}(\mathcal{S}, 2\delta)$, we only characterize the convergence rate assuming that the iterates have reached a neighborhood of a local minimum $x^\ast \in \mathcal{S}$. This assumption is not surprising; even some classical methods (e.g., Newton's method) achieve faster local rates when initialization is close enough to a local minimizer. Nevertheless, one can guarantee that iterates of Algorithm \ref{alg:retraction_free} can reach a neighborhood of a {\it stationary} point given the general convergence theorem in \cite{schneider2015convergence} as well as the global finite-time convergence of the landing algorithm in \cite{ablin2023infeasible}.

Without loss of generality, we assume $f_{\mathcal{S}}^* =0$ and omit $f_{\mathcal{S}}^*$ when applicable. We first present the following lemma to show that for any $x$ in the neighborhood of a local minimizer and in the safety region, the merit function is locally dominated by the landing field.

\begin{lemma}[Pseudo Gradient Domination \cite{sun2024global} ]\label{lem:pseudo-grad-dominate}
Let $x \in \St(d,r)^\epsilon \cap \mathcal{D}(\mathcal{S}, \delta) $, and without the loss of generality assume $f_{\mathcal{S}}^* =0$. Then, under Assumptions \ref{assump:Lip} and \ref{assump:PL}, we have,
\begin{equation}
\begin{aligned}
    \mathcal{L}(x) 
    &\leq \frac{1}{\mu'} \normfro{\Lambda(x)}^2,
\end{aligned}
\end{equation}
where $\ \ \frac{1}{\mu'} = \max\left\{\frac{1}{\mu}, \frac{2(3+{2}\epsilon)^2\hat{L}^2 + \mu L'}{2\mu \lambda^2(1-\epsilon)^2} \right\}.$
\end{lemma}
We note that $\mu'$ depends on a number of factors. In practice, if we set $\lambda = \mathcal{O}(L')$ and also $L'$ and $\hat{L}$ are of the same order, the two terms in $\mu'$ are of the same order, and we have $\mu' \approx \mu$.

Next, we use the following lemma to demonstrate the quadratic growth of the merit function $\cL(x)$ around the set of  local minima $\mathcal{S}$.

\begin{lemma}[Quadratic Growth] \label{lem:QG}
Let Assumptions \ref{assump:Lip} and \ref{assump:PL} hold. For $x \in \St(d,r)^\epsilon \cap \mathcal{D}(\mathcal{S}, \delta)$, the merit function $\cL(x)$ satisfies
\begin{align}\label{eq:quad}
    \cL(x) \geq \frac{\mu' \rho^2 }{4} \text{dist}(\mathcal{S}, x)^2.
\end{align}
\end{lemma}
\begin{proof}
    Since $x \in \St(d, r)^\epsilon$, given Proposition \ref{prop_L}, it is clear that $\normfro{\nabla \cL(x)} \geq \rho \normfro{ \Lambda(x)}.$
Combined with Lemma \ref{lem:pseudo-grad-dominate}, we have,
$$\mathcal{L}(x)  \leq \frac{1}{\mu'} \normfro{\Lambda(x)}^2 \leq \frac{1}{\mu' \rho^2} \normfro{\nabla \cL(x)}^2.$$
    Therefore, we know that the local Euclidean PŁ condition for $\cL(x)$ holds on $\forall x \in \mathcal{D}(\mathcal{S}, \delta) \cap \St(d, r)^\epsilon$. Given Proposition 2.2 in \cite{rebjock2023fast}, 
the local Euclidean PŁ condition on $\cL(x)$ implies the quadratic growth relationship \eqref{eq:quad}. 
\end{proof}

With the introduction of the previous two lemmas, we can now present our theorem. We show that when $x$ is in the neighborhood of a local minimizer, Algorithm \ref{alg:retraction_free} exhibits a local linear convergence under Assumption \ref{assump:PL}.

        \captionsetup{justification=centering}

\begin{figure*}[t]
    \centering
    \begin{subfigure}[b]{.2\textwidth}
\centering
\includegraphics[width=.895\textwidth]{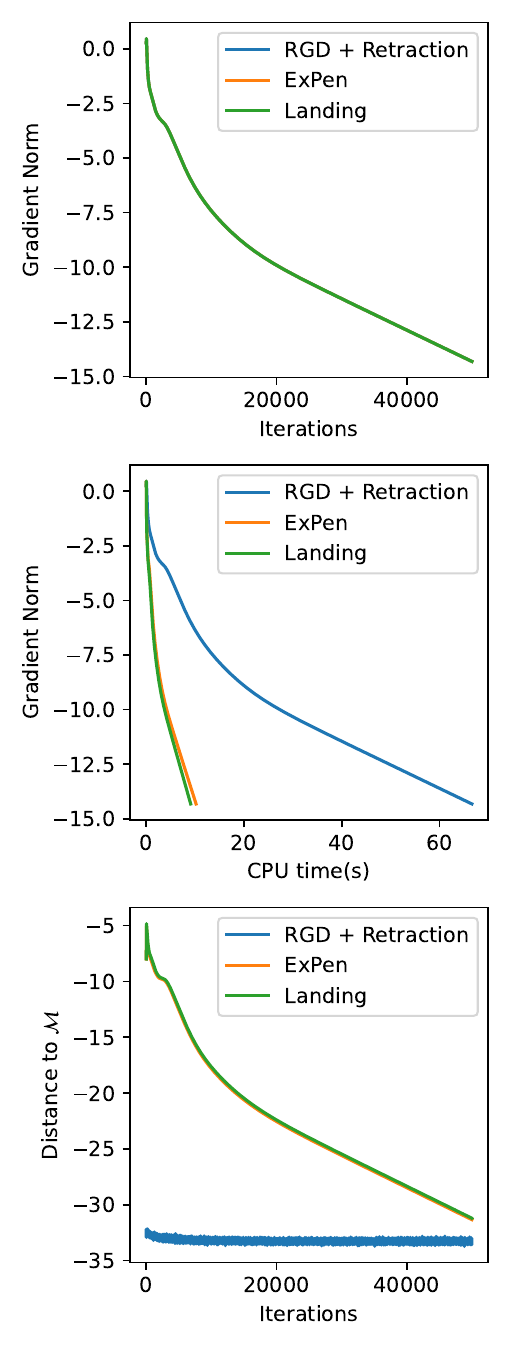}
    \caption{Performance of the landing algorithm on PCA.}
    \label{fig:PCA}
\end{subfigure}
\begin{subfigure}[b]{.79\textwidth}
\centering
        \includegraphics[width=\textwidth]{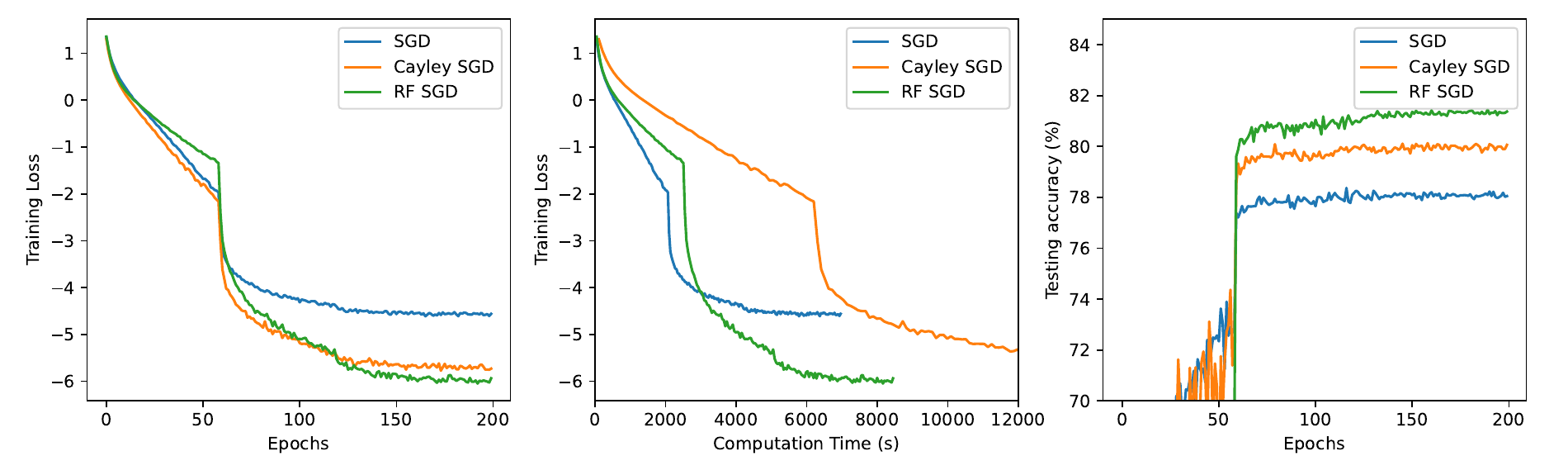}
        \includegraphics[width=\textwidth]{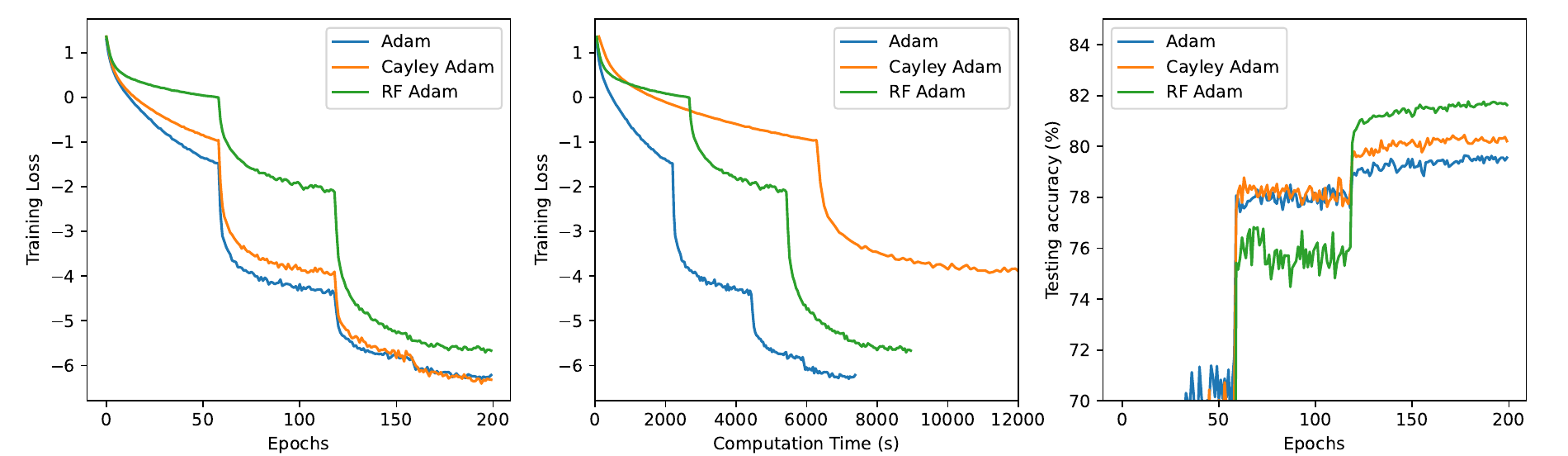}
    \caption{Comparison of retraction-free methods vs. SGD and Adam-type optimizers \\and their Cayley counterparts for ResNet on CIFAR-100 dataset.}
    \label{fig:CNN}\end{subfigure}
\end{figure*}

\begin{theorem}\label{central_thm}
    Let function $f$ satisfy Assumptions \ref{assump:Lip} and \ref{assump:PL}, and the initial point $x_0 \in \mathcal{D}(\mathcal{S}, \delta) \cap \St(d, r)^\epsilon$ is such that $\cL(x_0) \leq \frac{\mu' \rho^2 \delta^2}{16}$. Choose
    the step size $0<\alpha\leq \min\{\frac{\rho}{L'}, \alpha_{safe}\}$ and let iterates $\{x_k\}$ follow Algorithm \ref{alg:retraction_free}. Then, the merit function defined in \eqref{eq:merit_def} converges linearly $\forall k \geq 0$, such that
    \begin{equation}
        \cL(x_k)  \leq \left( 1 - \frac{\alpha \rho \mu'}{2 }\right)^k \cL(x_0).
    \end{equation}
\end{theorem}

\begin{proof}
We prove the theorem by induction. Suppose that $x_k \in \mathcal{D}(\mathcal{S}, \delta) \cap \St(d, r)^\epsilon$ and that $\cL(x_k) \leq \frac{\mu' \rho^2 \delta^2}{16}$ for some $k$. Since the merit function is  $L'$-Lipschitz smooth, using the update of the landing algorithm, we get
    \begin{equation*}
\resizebox{0.99\hsize}{!}{$
\begin{aligned}
\cL(x_{k+1})
\leq & \cL(x_{k}) + \langle \nabla \cL(x_k), x_{k+1} - x_{k}\rangle + \frac{L'}{2}\normfro{x_{k+1} - x_{k}}^2\\
    = & \cL(x_{k}) - \alpha \langle \nabla \cL(x_k), \Lambda(x_k)\rangle + \frac{\alpha^2 L'}{2}\normfro{\Lambda(x_k)}^2\\
    \leq & \cL(x_{k}) - \alpha \rho \normfro{\Lambda(x_k)}^2 + \frac{\alpha^2 L'}{2}\normfro{\Lambda(x_k)}^2\\
    \leq & (1 - \mu'(\alpha \rho - \frac{\alpha^2 L'}{2}))\cL(x_{k}),
  \end{aligned}
  $}
\end{equation*}
where the second to last line is from Proposition \ref{prop_L}, and the last inequality is from Lemma \ref{lem:pseudo-grad-dominate}. Now, if we choose a safe step size $0<\alpha\leq \min\{\frac{\rho}{L'}, \alpha_{safe}\}$, we get $\alpha \rho - \frac{\alpha^2 L'}{2} \geq \frac{\alpha \rho}{2}$.
This implies the monotonicity of $\cL$, and 
\begin{align}\label{eq:mono}
\cL(x_{k+1}) \leq \left( 1 - \frac{\alpha \rho \mu'}{2 }\right) \cL(x_{k}).    
\end{align}
Given that $x_k  \in \mathcal{D}(\mathcal{S}, \delta) \cap \St(d, r)^\epsilon$, from Lemma \ref{lem:QG} we get $\text{dist}(\mathcal{S}, x_k) \leq \sqrt{\frac{4 \cL(x_k)}{\mu' \rho^2}}\leq \frac{\delta}{2}$. Therefore, for the next iteration $k+1$, we have
\begin{align*}
    \text{dist}(\mathcal{S}, x_{k+1}) &\leq\text{dist}(\mathcal{S}, x_k) + \normfro{x_k - x_{k+1}}\\
    &=\text{dist}(\mathcal{S}, x_k) + \alpha \normfro{\Lambda(x_k)}\\
    &\leq \text{dist}(\mathcal{S}, x_k)+ \alpha L' \text{dist}(\mathcal{S}, x_k)\leq \delta.
\end{align*}
The last line stems from the Lipschitz continuity of $\Lambda(x)$ and the fact that $\alpha \leq \frac{\rho}{L'} \leq \frac{1}{2L'}$. This relationship combined with the safety guarantee ensures that for the next iteration, $x_{k+1} \in \mathcal{D}(\mathcal{S}, \delta) \cap \St(d, r)^\epsilon$ and that $\cL(x_{k+1}) \leq \frac{\mu' \rho^2 \delta^2}{16}$ due to \eqref{eq:mono}.

Therefore, given the initial condition on $x_0$ and by recursion, $x_k  \in \mathcal{D}(\mathcal{S}, \delta) \cap \St(d, r)^\epsilon$  and $\cL(x_k) \leq \frac{\mu' \rho^2 \delta^2}{16}$ is guaranteed to be satisfied by all $k$. Hence, we get $$\cL(x_k)  \leq \left( 1 - \frac{\alpha \rho \mu'}{2 }\right)^k \cL(x_0).$$
\end{proof}

Theorem \ref{central_thm} guarantees the exponential convergence of the merit function $\cL(x)$ while ensuring the iterates never leave the safe neighborhood region. Using Lemma \ref{lem:QG} and the Lipschitz continuity of $\Lambda(x)$, linear convergence rates for $\text{dist}(\mathcal{S}, x_{k})$ and $\normfro{\Lambda(x_k)}$ can be easily derived.

 \vspace{-2mm}

\section{Numerical Results} \label{sec:experiments}
For the experiments, we first evaluate the landing algorithm on a PCA task with synthetic data (Section \ref{sec:PCA}) and compare it with existing methods. Then, we evaluate it on training orthogonal-constrained CNNs (Section \ref{sec:CNN}).
 \vspace{-4mm}
\subsection{High-Dimensional PCA Task} \label{sec:PCA}
We start the numerical results with a traditional PCA task. The objective function for PCA can be written as:
\begin{equation*}
    \max_{x\in \R^{d\times r}} \langle A^\top A x,xD\rangle \ \ \text{s.t.} \  x\in \St(d, r),
\end{equation*}
where $A \in \R^{m\times d}$ is the data matrix and $D$ is a  diagonal matrix with $[D]_{11}>\cdots>[D]_{rr}>0$. 
We want to find the top-$r$ principal components of the data matrix. Let us set $d = 500, r = 20, m = 1000$ in our experiment.

We compare the performance of the landing algorithm with two other iterative gradient-based algorithms\footnote{ The code for the PCA experiment is available at \textit{https://github.com/sundave1998/Landing-Linear-Convergence}.}. We first consider a traditional Riemannian GD algorithm using projection as its retraction step. Secondly, we consider a state-of-the-art retraction-free approach named ExPen \cite{xiao2021solving}. All the algorithms are run with the same hyper-parameters and with the same initial points $x_0.$ The results are provided in Fig. \ref{fig:PCA}.

It is easy to verify that as long as the data matrix $A$ is not ill-conditioned and $rank(A) \geq r$, Assumption \ref{assump:PL} is satisfied \cite{liu2019quadratic}. We plot the dynamics of gradient norm $\norm{\grad f(x)}$ and the distance of iterates to $\mathcal{M}$, i.e., $\normfro{x^\top x - I_r}$, with respect to iterations as well as the CPU time. (i) All three algorithms exhibit similar linear convergence in terms of iterations, but the two retraction-free algorithms converge much faster in CPU time. (ii) Though the two retraction-free algorithms fail to strictly satisfy feasibility constraints initially, they eventually converge to a feasible critical point on $\St(d,r)$. (iii) While the landing algorithm performs similarly to ExPen, it does not require calculating the gradient of an alternative point other than $x$; not only is it more efficient, but also it is more suitable for implementation in modern neural networks, which we present in the next experiment.
 \vspace{-3.4mm}

\subsection{Training CNNs without Retractions}\label{sec:CNN}
We adopt the experimental setting in \cite{li2020efficient} to compare retraction-free algorithms with the state-of-the-art \textit{Cayley SGD} and \textit{Cayley Adam}, as well as vanilla SGD and Adam, which are suitable for unconstrained Euclidean problems.
Although no strict function assumptions are satisfied by neural networks due to the nature of ReLU activations, some works \cite{bai2019beyond} suggested that the loss of wide neural networks can be coupled with that of a quadratic model.
We refer to the retraction-free methods as RF-SGD and RF-Adam, respectively. For the adaptation of RF-Adam, we refer to Algorithm 2 in \cite{li2020efficient} for details but replace the \textit{Cayley\_loop} with a retraction-free landing update.

We evaluate all algorithms with Wide ResNet (WRN-28-10) on the CIFAR-100 dataset, where generalization is an important metric. All hyper-parameters are adapted from \cite{li2020efficient} with no modifications.
We present the results in Fig. \ref{fig:CNN}, suggesting that orthogonality constraints (imposed in Riemannian optimization) may help with better generalization, a phenomenon also observed in other works (e.g.,  \cite{cogswell2015reducing}). 
We note that the mini-batch training in neural networks introduces stochastic noise in gradient evaluations, negatively impacting convergence. In our experiments, all three methods are run with a learning rate decay scheduler in accordance with the previous literature, resulting in a multi-stage convergence with jumps every 50 epochs. 
In terms of generalization, both RF-SGD and RF-Adam perform better than their Cayley counterparts \cite{li2020efficient}, which suggests that compared to inexact Cayley updates, RF-based algorithms are perhaps more accurate and/or have more desirable properties in training neural networks.

We also plot the convergence of loss vs. computation time. Retraction-free (RF) methods have much faster convergence compared to Cayley-retraction methods. We note that the update in \cite{li2020efficient} is already an approximate solution with emphasis on efficiency, yet the landing algorithm helps with faster convergence and moderately improves the generalization.

\vspace{-2mm}

\section{Conclusion}
In this work, we studied a faster convergence for the retraction-free landing algorithm, addressing optimization under orthogonality constraints. We used the local Riemannian PŁ condition on the objective function to derive a local linear convergence rate for the landing algorithm. We also provided numerical results as a verification of the convergence speed while highlighting the computational efficiency of the algorithm compared to existing benchmarks.

For future research, generalization of these techniques to general manifold constraints, such as the Grassmann and Hyperbolic manifolds, is an interesting direction. In addition, given the wide application of orthogonality constraints in modern neural network structures, developing adaptive and momentum-based retraction-free optimization algorithms is another interesting open question.